\documentclass[12pt,letterpaper]{article}
\usepackage{calc}
\usepackage{array}
\usepackage{graphicx}
\usepackage{subfigure}
\usepackage[indent,bf]{caption}
\usepackage{rotating}
\usepackage{setspace}
\usepackage{longtable}
\usepackage{rotating}
\usepackage{amssymb, amsmath, amsfonts, amsthm, epsfig}
\usepackage{epsfig,latexsym}
\usepackage{multicol}
\usepackage{pgf,tikz}
\usepackage{verbatim}
\usepackage{url}
\usepackage{fancyhdr}
\usetikzlibrary{arrows}
\usepackage[dvips,%
bookmarksopen=false,
pdfpagemode=none % don't start with the navigation pane open in Acrobat
]{hyperref}

\begin{document}
\sloppy

\newfont{\blb}{msbm10 scaled\magstep1}
\newtheorem{lem}{Lemma}[section]
\newtheorem{ex}{Example}[section]
\newtheorem{cor}[lem]{Corollary}
\newtheorem{claim}[lem]{Claim}
\newtheorem{thm}[lem]{Theorem}
\newtheorem*{123}{1-2-3 Conjecture}
\newtheorem*{12}{1-2 Conjecture}
\newtheorem*{list123}{List 1-2-3 Conjecture}
\newtheorem*{LLL}{Lov\'asz Local Lemma}
\newtheorem*{SLL}{Symmetric Local Lemma}
\newtheorem*{MLL}{Modified Local Lemma}

\newtheorem{prop}[lem]{Proposition}
\newtheorem{prob}{Problem}
\newtheorem{conj}{Conjecture}
\newtheorem{defn}{Definition}

\newcommand{\lmulti}{\{\hspace{-0.035in}\{}
\newcommand{\rmulti}{\}\hspace{-0.035in}\}}

\newcommand{\per}{\textup{\rm per}\,}
\newcommand{\mind}{\textup{\rm mind}}
\newcommand{\tmind}{\textup{\rm tmind}}
\newcommand{\pind}{\textup{\rm pind}}
\newcommand{\ch}{\textup{\rm ch}}
\newcommand{\M}{\mathbb{M}}
\newcommand{\N}{\mathbb{N}}
\newcommand{\Z}{\mathbb{Z}}
\newcommand{\R}{\mathbb{R}}
\newcommand{\C}{\mathbb{C}}
\newcommand{\F}{\mathbb{F}}
\newcommand{\pr}{\mathbb{P}}
\newcommand{\D}{\Delta}
\newcommand{\Ls}{\mathcal{L}}
\newcommand{\PP}{\mathcal{P}}

\newcommand{\Se}{\chi_\Sigma^e}
\newcommand{\Pe}{\chi_\Pi^e}
\newcommand{\me}{\chi_m^e}
\newcommand{\se}{\chi_s^e}
\newcommand{\Sv}{\chi_\Sigma^v}
\newcommand{\Pv}{\chi_\Pi^v}
\newcommand{\mv}{\chi_m^v}
\newcommand{\sv}{\chi_s^v}
\newcommand{\St}{\chi_\Sigma^t}
\newcommand{\Pt}{\chi_\Pi^t}
\newcommand{\mt}{\chi_m^t}
\newcommand{\st}{\chi_s^t}
\newcommand{\chSe}{\ch_\Sigma^e}
\newcommand{\chPe}{\ch_ \Pi ^e}
\newcommand{\chme}{\ch_m^e}
\newcommand{\chse}{\ch_s^e}
\newcommand{\chSv}{\ch_\Sigma^v}
\newcommand{\chPv}{\ch_ \Pi ^v}
\newcommand{\chmv}{\ch_m^v}
\newcommand{\chsv}{\ch_s^v}
\newcommand{\chSt}{\ch_\Sigma^t}
\newcommand{\chPt}{\ch_ \Pi ^t}
\newcommand{\chmt}{\ch_m^t}
\newcommand{\chst}{\ch_s^t}

\newcommand{\eSe}{{\chi'}_\Sigma^e}
\newcommand{\ePe}{{\chi'}_\Pi^e}
\newcommand{\eme}{{\chi'}_m^e}
\newcommand{\ese}{{\chi'}_s^e}
\newcommand{\eSv}{{\chi'}_\Sigma^v}
\newcommand{\ePv}{{\chi'}_\Pi^v}
\newcommand{\emv}{{\chi'}_m^v}
\newcommand{\esv}{{\chi'}_s^v}
\newcommand{\eSt}{{\chi'}_\Sigma^t}
\newcommand{\ePt}{{\chi'}_\Pi^t}
\newcommand{\emt}{{\chi'}_m^t}
\newcommand{\est}{{\chi'}_s^t}
\newcommand{\echSe}{{\ch'}_\Sigma^e}
\newcommand{\echPe}{{\ch'}_ \Pi ^e}
\newcommand{\echme}{{\ch'}_m^e}
\newcommand{\echse}{{\ch'}_s^e}
\newcommand{\echSv}{{\ch'}_\Sigma^v}
\newcommand{\echPv}{{\ch'}_ \Pi ^v}
\newcommand{\echmv}{{\ch'}_m^v}
\newcommand{\echsv}{{\ch'}_s^v}
\newcommand{\echSt}{{\ch'}_\Sigma^t}
\newcommand{\echPt}{{\ch'}_ \Pi ^t}
\newcommand{\echmt}{{\ch'}_m^t}
\newcommand{\echst}{{\ch'}_s^t}

\newcommand{\tSe}{{\chi''}_\Sigma^e}
\newcommand{\tPe}{{\chi''}_\Pi^e}
\newcommand{\tme}{{\chi''}_m^e}
\newcommand{\tse}{{\chi''}_s^e}
\newcommand{\tSv}{{\chi''}_\Sigma^v}
\newcommand{\tPv}{{\chi''}_\Pi^v}
\newcommand{\tmv}{{\chi''}_m^v}
\newcommand{\tsv}{{\chi''}_s^v}
\newcommand{\tSt}{{\chi''}_\Sigma^t}
\newcommand{\tPt}{{\chi''}_\Pi^t}
\newcommand{\tmt}{{\chi''}_m^t}
\newcommand{\tst}{{\chi''}_s^t}
\newcommand{\tchSe}{{\ch''}_\Sigma^e}
\newcommand{\tchPe}{{\ch''}_\Pi ^e}
\newcommand{\tchme}{{\ch''}_m^e}
\newcommand{\tchse}{{\ch''}_s^e}
\newcommand{\tchSv}{{\ch''}_\Sigma^v}
\newcommand{\tchPv}{{\ch''}_\Pi ^v}
\newcommand{\tchmv}{{\ch''}_m^v}
\newcommand{\tchsv}{{\ch''}_s^v}
\newcommand{\tchSt}{{\ch''}_\Sigma^t}
\newcommand{\tchPt}{{\ch''}_\Pi ^t}
\newcommand{\tchmt}{{\ch''}_m^t}
\newcommand{\tchst}{{\ch''}_s^t}

\newcommand{\We}{\chi_{\sigma}^e}
\newcommand{\we}{\chi_{\sigma^*}^e}
\newcommand{\Wt}{\chi_{\sigma}^t}
\newcommand{\wt}{\chi_{\sigma^*}^t}
\newcommand{\Wv}{\chi_{\sigma}^v}
\newcommand{\wv}{\chi_{\sigma^*}^v}
\newcommand{\chWe}{\ch_{\sigma}^e}
\newcommand{\chwe}{\ch_{\sigma^*}^e}
\newcommand{\chWt}{\ch_{\sigma}^t}
\newcommand{\chwt}{\ch_{\sigma^*}^t}

\newcommand{\KLT}{\textrm{Karo{\'n}ski-{\L}uczak-Thomason}}

\newcommand{\spc}{\hspace{0.08in}}

\renewcommand{\thepage}{\arabic{page}}

%\begin{frontmatter}

\title{Sequence variations of the 1-2-3 Conjecture and irregularity strength}

%\author{Ben Seamone%\corref{cor1}}
%\and
%Brett Stevens}
%%\ead{bseamone@connect.carleton.ca}
%%% \ead[url]{home page}
%%\cortext[cor1]{Corresponding author}
%%\author{Brett Stevens}
%%\ead{brett@math.carleton.ca}
%%\address{Carleton University}
%%\address{School of Mathematics and Statistics, 1125 Colonel By Drive, Ottawa, ON, K1S 5B6 Canada}

\author{
Ben Seamone\footnote{School of Mathematics and Statistics,
Carleton University, Ottawa, Canada
\texttt{\{bseamone,brett\}@math.carleton.ca}}
\and
Brett Stevens\footnotemark[1]
}

\maketitle

\begin{abstract}
Karo{\'n}ski, {\L}uczak, and Thomason (2004) conjectured that, for any connected graph $G$ on at least three vertices, there exists an edge weighting from $\{1,2,3\}$ such that adjacent vertices receive different sums of incident edge weights.  Bartnicki, Grytczuk, and Niwcyk (2009) made a stronger conjecture, that each edge's weight may be chosen from an arbitrary list of size $3$ rather than $\{1,2,3\}$.  We examine a variation of these conjectures, where each vertex is coloured with a {\em sequence} of edge weights.  Such a colouring relies on an ordering of $E(G)$, and so two variations arise -- one where we may choose any ordering of $E(G)$ and one where the ordering is fixed.  In the former case, we bound the list size required for any graph.  In the latter, we obtain a bound on list sizes for graphs with sufficiently large minimum degree.  We also extend our methods to a list variation of irregularity strength, where each vertex receives a distinct sequence of edge weights.
\end{abstract}

%\begin{keyword}
%%% keywords here, in the form: keyword \sep keyword
%vertex colouring \sep edge weighting \sep total weighting \sep irregularity strength \sep graph labeling  \sep 1-2-3 Conjecture
%%% MSC codes here, in the form: \MSC code \sep code
%%% or \MSC[2008] code \sep code (2000 is the default)

%\end{keyword}

%\end{frontmatter}

\section{Introduction and Brief Survey}

A graph $G = (V,E)$ will be simple and loopless unless otherwise stated.  Throughout, we write $[k]$ for the set $\{1, 2,\ldots,k\}$.  An {\bf edge $k$-weighting}, $w$, of $G$ is a an assignment of a number from $[k]$ to each $e \in E(G)$, that is $w: E(G) \rightarrow[k]$.  Karo{\'n}ski, {\L}uczak, and Thomason \cite{KLT04} conjectured that, for every graph without a component isomorphic to $K_2$, there is an edge $3$-weighting such that any two adjacent vertices have different sums of incident edge weights.  If an edge $k$-weighting gives rise to such a proper vertex colouring, we say that the weighting is a {\bf vertex colouring by sums}.  We will denote by $\Se(G)$ the smallest value of $k$ such that a graph $G$ has an edge $k$-weighting which is a vertex colouring by sums (this notation is a slight modification of that proposed by Gy{\H o}ri and Palmer in \cite{GP09}).  We say that a graph $G$ is {\bf nice} if no component is isomorphic to $K_2$.  We may express Karo{\'n}ski, {\L}uczak, and Thomason's conjecture (frequently called the ``1-2-3 Conjecture") as follows:

\begin{123}[Karo{\'n}ski, {\L}uczak, Thomason \cite{KLT04}] 
If $G$ is nice, then $\Se(G) \leq 3$.
\end{123}

One may also obtain a vertex colouring from an edge $k$-weighting by considering the products, sets, or multisets of incident edge weights.  The smallest $k$ for which a graph $G$ has an edge $k$-weighting which is a proper vertex colouring by products, sets or multisets will be denoted $\Pe(G), \se(G)$ and $\me(G)$, respectively.  The best known bounds for these graph parameters are, for any nice graph $G$, $\Se(G) \leq 5$ \cite{KKP1}, $\me(G) \leq 4$ \cite{AADR05}, $\Pe(G) \leq 5$ \cite{SK08}, and $\se(G) =  \lceil{\log_2{\chi(G)}}\rceil + 1$ \cite{GP09}.  It is shown in  \cite{AADR05} that if $\delta(G) \geq 1000$, then $\me(G) \leq 3$.  In \cite{ADR08} it is shown that, asymptotically almost surely, $\Se(G) \leq 2$.

One may also allow each vertex to receive a weight from $[k]$, in addition to the edge weights; such weightings of $G$ are called {\bf total $k$-weightings}.  Vertex colourings via total weightings are obtained by considering the weights of the edges incident to a vertex as well as the vertex's weight itself.  The smallest $k$ for which a graph $G$ has a total $k$-weighting which is a proper vertex colouring by sums, products, sets or multisets is denoted $\St(G), \Pt(G), \st(G)$ and $\mt(G)$, respectively.  

The following conjecture motivates the study of total weightings and vertex colouring by sums:

\begin{12}[Przyby{\l}o, Wo\'zniak \cite{PW10}] 
For every graph $G$, $\St(G) \leq 2$.
\end{12}

Clearly, any upper bound on an edge $k$-weighting parameter is an upper bound on its corresponding total $k$-weighting parameter.  The best known improvements on the bounds above are, for an arbitrary graph $G$, $\St(G) \leq 3$ \cite{Kal} (in fact, only vertex weights 1 and 2 are required) and $\Pt(G) \leq 3$ \cite{SK08}.  Clearly both $\St(G)$ and $\Pt(G)$ are upper bounds on $\mt(G)$, so we have that $\mt(G) \leq 3$ as well.

All of the above graph colouring parameters have natural list generalizations.  Rather than choosing a weight from $[k]$ for each edge (vertex), one must choose a weight for each edge (vertex) from a set of $k$ arbitrary real numbers independently assigned to each edge (vertex).  We call such weightings {\bf edge $k$-list-weightings} and {\bf total $k$-list-weightings} (in the case where vertex weights are included).  Given a graph $G$, the smallest $k$ such that any assignment of lists of size $k$ to $E(G)$ permits an edge $k$-list-weighting which is a vertex colouring by sums is denoted $\chSe(G)$; each of the parameters above generalizes similarly.

The following conjecture proposes a stronger version of the 1-2-3 Conjecture:

\begin{list123}[Bartnicki, Grytczuk, Niwcyk \cite{BGN09}] 
If $G$ is a nice graph, then $\chSe(G) \leq 3$.
\end{list123}

%The published results on $k$-edge list weightings establish particular classes of graphs for which a bound on $\chSe(G)$ exists; 
It is shown in \cite{mythesis, Ben-CN} that $\chSe(G) \leq 2\Delta(G) + 1$ for any nice graph $G$.  However, there is no known integer $K$ such that $\chSe(G) \leq K$ for any nice graph $G$.  Bartnicki et al. \cite{BGN09} establish that $\chSe(G) \leq 3$ if $G$ is complete, complete bipartite, or a tree.  The analogous problem for digraphs is also solved in \cite{BGN09} and \cite{Ben2}.  In the former, a constructive method is used to show that $\chSe(D) \leq 2$ for any digraph $D$; the latter provides an alternate proof using algebraic methods.

The multiset version of the 1-2-3 Conjecture is a natural relaxation of the requirement that adjacent vertices receive distinct sums.  This paper is concerned with a further relaxation of the multiset version, where one requires that adjacent vertices receive distinct sequences (given some reasonable method of constructing a sequence from weights of incident edges).  In Section 2, we introduce the problem of colouring $V(G)$ by sequences of weights from incident edges.  In Section 3, we study colouring by sequences with the requirement that {\em every} vertex receives a distinct sequence rather than only adjacent vertices; this is a variation of a well studied parameter known as the irregularity strength of a graph.  Wherever possible, we study the stronger ``list versions'' of these weighting problems.

\section{Vertex Colouring by Sequences}

We must first define how to induce a sequence of weights from an edge weighting.  
%One natural way to do this is to define an ordering of the edge set of a graph and use it to determine the order of the elements seen by a vertex.  More precisely, let 
Let $E(G) = \{e_1, e_2, \ldots, e_m\}$ be the edge set of a graph $G$, $\prec$ a total order on $E(G)$, and let $w:E(G) \to S$ be an edge weighting of $G$.  For a vertex $v \in V(G)$, let $I_v = \{i \,:\, e_i \ni v\}$.  A {\bf colouring of $V(G)$ from $w$ by sequences} is obtained by constructing a sequence for each $v \in V(G)$ by taking the multiset $\lmulti w(e_i) \,:\, i \in I_v\rmulti$ and ordering the elements according to $w(e_i) \prec w(e_j)$ if and only if $e_i \prec e_j$.  

For example, consider $C_5$ with vertices and edges labelled as in Figure \ref{C5seq}: %on page \pageref{C5seq}.  
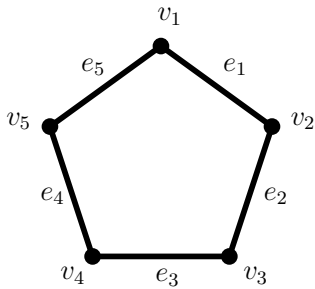
\begin{figure}[htb]
\begin{center}
\scalebox{0.9}{
\begin{tikzpicture}%[line cap=round,line join=round,>=triangle 45,x=1.0cm,y=1.0cm]
\clip(-3,-0.5) rectangle (3,3.8);
\draw [line width=2.4pt] (-1.01,0)-- (1.01,0);
\draw [line width=2.4pt] (1.01,0)-- (1.64,1.92);
\draw [line width=2.4pt] (1.64,1.92)-- (0,3.11);
\draw [line width=2.4pt] (0,3.11)-- (-1.64,1.92);
\draw [line width=2.4pt] (-1.64,1.92)-- (-1.01,0);
\fill [color=black] (-1.01,0) circle (3.5pt);
\draw[color=black] (-1.3,-0.28) node {$v_4$};
\fill [color=black] (1.01,0) circle (3.5pt);
\draw[color=black] (1.4,-0.28) node {$v_3$};
\draw[color=black] (0.1,-0.31) node {$e_3$};
\draw[color=black] (1.7,0.9) node {$e_2$};
\draw[color=black] (1.1,2.8) node {$e_1$};
\draw[color=black] (-1,2.8) node {$e_5$};
\draw[color=black] (-1.6,0.9) node {$e_4$};
\fill [color=black] (1.64,1.92) circle (3.5pt);
\draw[color=black] (2.1,2) node {$v_2$};
\fill [color=black] (0,3.11) circle (3.5pt);
\draw[color=black] (0.13,3.55) node {$v_1$};
\fill [color=black] (-1.64,1.92) circle (3.5pt);
\draw[color=black] (-2.1,2) node {$v_5$};
\end{tikzpicture}
}
\end{center}
\vspace{-0.2in}
\caption{A labelled $5$-cycle}
\label{C5seq}
\end{figure}

We will consider two edge orderings of this graph and attempt to properly colour the vertices by sequences for each using only two edge weights.

%Consider edge weightings of $C_5$ using only weights $a$ and $b$.  
If the edges are ordered $e_1 \prec e_2 \prec e_3 \prec e_4 \prec e_5$, then there is an edge $2$-weighting $w:E(C_5) \to \{a,b\}$ which is a proper colouring by sequences, given in Table \ref{c5weighting}: \\

\begin{table}[h!]
\begin{center}
	\begin{tabular}{| c | c |}
    	\hline
	Weighting & Colouring \\
	\hline
	$w(e_1) = a$ & $c(v_1) = aa$ \\
	$w(e_2) = b$ & $c(v_2) = ab$ \\
	$w(e_3) = a$ & $c(v_3) = ba$ \\
	$w(e_4) = b$ & $c(v_4) = ab$ \\ 
	$w(e_5) = a$ & $c(v_5) = ba$ \\
		\hline
  	\end{tabular} \\
  \caption{An edge $2$-weighting that properly colours $V(C_5)$ by sequences}
  \label{c5weighting}
\end{center}
\end{table} 

However, for the ordering $e_1 \prec e_3 \prec e_5 \prec e_2 \prec e_4$, the vertex colours given by a weighting $w$ are
	$$\begin{array}{c}
	c(v_1) = w(e_1)w(e_5), \\
	c(v_2) = w(e_1)w(e_2), \\
	c(v_3) = w(e_3)w(e_2), \\
	c(v_4) = w(e_3)w(e_4), \\
	c(v_5) = w(e_5)w(e_4).
	\end{array} $$
In order to have a proper colouring, 
	\begin{eqnarray}
	w(e_5) &\neq& w(e_2), \\
	w(e_1) &\neq& w(e_3), \\
	w(e_2) &\neq& w(e_4), \\
	w(e_3) &\neq& w(e_5). 
	\end{eqnarray}
If $w$ is a weighting with only two edge weights, then inequalities (1) and (3) imply that $w(e_4) = w(e_5)$, while (2) and (4) imply that $w(e_1) = w(e_5)$.  Together, this forces $c(v_1) = c(v_5)$, and hence $C_5$ cannot be properly vertex coloured by sequences with two edge weights for the ordering $e_1 \prec e_3 \prec e_5 \prec e_2 \prec e_4$.

Hence, the order of the edges plays a significant role in vertex colouring by sequences.  As such, we consider the following two problems:

\begin{prob}\label{prob1}
Given a graph $G$, what is the smallest value of $k$ such that there is an edge $k$-weighting of $G$ which gives a proper colouring of $V(G)$ by sequences {\em for some} ordering of $E(G)$?
\end{prob}

\begin{prob}\label{prob2}
Given a graph $G$, what is the smallest value of $k$ such that there is an edge $k$-weighting of $G$ which gives a proper colouring of $V(G)$ by sequences {\em for every} ordering of $E(G)$?
\end{prob}

These parameters will be called $\we(G)$ and $\We(G)$, respectively.  For the list-weighting variations, $\chwe(G)$ and $\chWe(G)$ will be used.

\subsection{Colouring by sequences for some $E(G)$ ordering}

The case when one is free to choose an ``optimal'' ordering of the edges of a graph $G$ is the easier of the two problems to analyze.  In this section, Problem \ref{prob1} is solved completely for edge weightings and total weightings for graphs and multigraphs.

We begin with the simple case of cycles.

\begin{prop}\label{wordcycle}
If $n \geq 3$, $\chwe(C_n) = 2$.
\end{prop}

\begin{proof}
The $n=3$ case is trivial.  Let $n \geq 4$, $V(C_n) = \{v_1, \ldots, v_n\}$, and $E(C_n) = \{v_iv_{i+1} \,:\, 1 \leq i \leq n\}$ with index addition taken $\bmod{\,\,n}$.  Let $e_i = v_iv_{i+1}$. 
%Consider the edges in the order they appear around the cycle, say $e_1, \ldots, e_n$.  
For each $i = 1, \ldots, n$, let $L_{e_i}$ be a set of 2 elements.   Choose $w(e_2) \in L_{e_2}$ and $w(e_n) \in L_{e_n}$ such that $w(e_2) \neq w(e_n)$.  For each $i = 3, \ldots, n-1$, let $w(e_i) \in L_{e_i} \setminus \{w(e_{i-1})\}$ and let $w(e_1) \in L_{e_1} \setminus \{w(e_{n-1})\}$.  The resulting vertex colouring by sequences is proper.
\end{proof}

The {\bf prefix} of length $t$ of a sequence $a_1a_2\cdots a_n$ is the subsequence $a_1a_2\cdots a_t$.  A vertex colouring by sequences, $c$, is {\bf prefix distinguishing} if, for any $uv \in E(G)$ with $d(u) \geq d(v) \geq 2$, $c(v)$ is not the prefix of $c(u)$; in other words, if $c(v) = a_1a_2\cdots a_k$ and $c(u) = b_1b_2\cdots b_l$ for some $l \geq k$, then there exists an index $i$, $1 \leq i \leq k$, such that $a_i \neq b_i$.  
Clearly any prefix distinguishing vertex colouring by sequences is also a proper vertex colouring.  
By proving a stronger statement about prefix distinguishing colourings by sequences, one can show that $\chwe(G) \leq 2$ for every nice graph $G$.  Note that we use $|S|$ to denote the length of a sequence $S$.

\begin{thm}\label{word}
Let $G$ be a nice connected graph and for each $e \in E(G)$ let $L_e$ be a set of two real numbers.  There is an ordering of $E(G)$ and values $w(e) \in L_e$, $e \in E(G)$, such that $w$ is a prefix distinguishing vertex colouring by sequences.
\end{thm}

\begin{proof}
We prove the statement by induction on $|V(G)|$.  The theorem is true if $|V(G)| = 3$; assume $|V(G)| \geq 4$.  Let $d = \delta(G)$, let $x \in V(G)$ be a vertex of minimum degree, and let $G' = G - x$ (note that no component of $G'$ is isomorphic to $K_2$).  For various values of $d$, it will be shown that an edge weighting $w'$ which gives a prefix distinguishing vertex colouring by sequences of $G'$, say $c'$, 
%given by $w$ restricted to $E(G')$ 
can be extended to $G$.  Let $w$ and $c$ denote the extended edge weighting and vertex colouring of $G$, respectively.  In each case we consider, the ordering of the edges of $E(G')$ which gives $c'$ is extended by appending the edges incident to $x$ to the end of the ordering (and hence, the weights of these edges to the ends of the colours of vertices in $N_G(x)$).

Suppose $d = 1$, and let $N_G(x) = \{y\}$.  If $d_{G'}(y) = 1$, let $z$ be the neighbour of $y$ in $G'$ and choose $w(xy) \in L_{xy}$ different from the second entry in $c'(z)$.  Otherwise, $d_{G'}(y) \geq 2$ and so, since $c'$ is prefix distinguishing, any choice of $w(xy) \in L_{xy}$ gives a prefix distinguishing colouring of $V(G)$.

Suppose $d=2$.  If $G$ is a cycle, then the result follows by Proposition \ref{wordcycle}.  Assume $G$ is not a cycle.  One may choose $x$ such that one of its neighbours has degree at least 3 in $G$; call this neighbour $y_1$.  Let $y_2$ denote the other neighbour of $x$.  There are two forbidden values of $c(x)$ given by the length 2 prefixes of $c'(y_1)$ and $c'(y_2)$.  If $d_{G'}(y_2) = 1$, let $z$ denote the neighbour of $y_2$ in $G'$ and choose $w(xy_2) \in L_{xy_2}$ different from the second entry in $c'(z)$.  There are then at least three possible colours for $c(x)$, and so at least one permissible choice of $w(xy_1) \in L_{xy_1}$.  Suppose that $d_{G'}(y_2) \geq 2$.  There are at least four possible colours for $c(x)$, and so at least one permissible choice of $w(xy_1) \in L_{xy_1}$ and $w(xy_2) \in L_{xy_2}$, and hence at least two permissible choices which give the desired $c$.

Suppose $d \geq 3$, and let $N_G(x) = \{y_1, \ldots y_d\}$.  Order $E(G)$ beginning with the edges of $E(G')$ as ordered by the induction hypothesis, and adding $xy_1 \prec \ldots \prec xy_d$ to the end of the ordering.  Since $c'$ is prefix distinguishing in $G'$, any choices of $w(xy_i) \in L_{xy_i}$, $i = 1, \ldots, d$, will be prefix distinguishing in $G$ except perhaps between $x$ and some $y_i$.  Since $\delta(G') \geq d-1$, the length of each sequence $c'(y_i)$ is at least $d-1$.  Forbid $x$ from receiving the same $(d-1)$-prefix as any of $y_1, y_2, \ldots, y_d$.  
% (there are at most $d$ such prefixes).  
There are $2^{d-1}$ choices for the weights of $xy_1, \ldots xy_{d-1}$, and hence for the prefix of length $d-1$ of $c(x)$.  Since $d \geq 3$, it follows that $2^{d-1} > d$ and hence at least one $(d-1)$-prefix does not conflict with any of the $(d-1)$-prefixes of the colours assigned to $y_1, y_2, \ldots, y_d$.  Any choice of $xy_d \in L_{xy_d}$ completes the weighting.
\end{proof}

\begin{cor}\label{chooseorder}
If $G$ is a nice graph, then $\chwe(G) \leq 2$.
\end{cor}

To obtain a similar result for a total $k$-weighting of a graph $G$, create a new graph $H$ by adding a leaf to each $v \in V(G)$ and assigning the new leaf edge incident to $v$ the list $L_v$.  Applying Theorem \ref{word} to $H$ gives an ordering of the vertices and edges of $G$ and a total $k$-list-weighting of $G$ which colours $V(G)$ by sequences.  Hence, we have the following corollary:

\begin{cor}\label{totalword}
For any graph $G$, $\chwt(G) \leq 2$.
\end{cor}

Theorem \ref{word} also easily extends to multigraphs.  We call a multigraph {\bf nice} if it has no loopless connected component with exactly two vertices.

\begin{thm}
If $M$ is a nice, loopless multigraph, then there is an ordering of $E(M)$ such that, for any assignment of lists of size $2$ to the edges of $M$, there exists an edge $2$-list-weighting $w$ which gives a prefix distinguishing vertex colouring by sequences.\end{thm}

\begin{proof}
Let $G$ be the underlying simple graph of $M$ and apply Theorem \ref{word} to $G$; denote by $c$ the resulting colouring of $G$.  The edges of $E(M) \setminus E(G)$ will be assigned to the end of the ordering of $E(G)$.  For every $uv \in E(M)$ such that $d_G(u) \geq d_G(v) \geq 2$, any assignment of weights to edges in $E(M) \setminus E(G)$ preserves the prefix distinguishing vertex colouring.  Consider $x \in V(G)$ with $d_G(x) = 1$.  If $d_M(x) = 1$ as well, then any assignment of weights to the remaining edges will preserve the prefix-distinguishing vertex colouring.  If $d_M(x) \geq 2$, then let $e \in E(M) \setminus E(G)$ be incident to $x$ and let $y$ be the other end of $e$.  Choosing $w(e)$ different from the second entry in $c(y)$ preserves the prefix-distinguishing vertex colouring.
\end{proof}

The following corollary follows in the same manner as Corollary \ref{totalword}.

\begin{cor}
If $M$ is a nice multigraph, then there is an ordering of $E(M) \cup V(M)$ such that, for any assignment of lists of size $2$ to the edges and vertices of $M$, there exists a total $2$-list-weighting $w$ which gives a prefix distinguishing vertex colouring by sequences.
\end{cor}

\subsection{Colouring by sequences for any $E(G)$ ordering}

We now turn our attention to the problem of determining $\We(G)$, $\Wt(G)$, $\chWe(G)$, and $\chWt(G)$ for a graph $G$.  Each bound is clearly bounded above by its multiset counterpart (i.e. $\We(G) \leq \me(G))$.  
In general, it is not clear for which graphs these bounds are tight.  For example, $\We(C_3) = 2$ and $\me(C_3) = 3$, whereas $\We(C_4) = \me(C_4) = 2$ and $\We(C_5) = \me(C_5) = 3$.

The following proposition follows from two bounds stated in the introduction -- $\me(G) \leq 4$ for every nice graph and $\St(G) \leq 3$ for every graph.

\begin{prop}
If a graph $G$ is nice then $\We(G) \leq 4$.  For any graph $G$, $\Wt(G) \leq 3$.
\end{prop}

%There are no known constant bounds (i.e. not depending on $G$) for $\chWe(G)$ or $\chWt(G)$.  However, since total-weightings of a graph $G$ can be thought of as edge-weightings of the graph obtained by adding a leaf to every vertex of $G$, we have that $\Wt(G) \leq \We(G)$ and $\chWt(G) \leq \chWe(G)$.  Furthermore, since $\textrm{ch}(G) \geq \chSt(G)$, we can see that for any $G$, $\textrm{ch}(G)$ is also an upper bound for $\chWt(G)$.

We make the following conjectures, in light of the conjectures stated in the opening section:

\begin{conj}\label{We}
If $G$ is a nice graph then $\chWe(G) \leq 3$.
\end{conj}

\begin{conj}\label{Wt}
For any graph $G$, $\chWt(G) \leq 2$.
\end{conj}

%Clearly, the 1-2-3 Conjecture and 1-2 Conjecture would imply the weaker statements that $\We(G) \leq 3$ for nice graphs and $\Wt(G) \leq 2$ for all graphs.  
Conjectures \ref{We} and \ref{Wt} are verified here for $d$-regular graphs of sufficiently large degree and for general graphs with $\delta(G)$ sufficiently large in terms of $\Delta(G)$.

We begin with a few necessary definitions.   
For a set of events $\{A_i : i \in I\}$ in a probability space and a subset $K \subseteq I$, define $A_K := \bigcap_{i \in K} A_i$ and $\overline{A}_K := \bigcap_{i \in K} \overline{A}_i$.
%Let $\{A_i : i \in I\}$ be events in a probability space and 
Let $J \subseteq I \setminus \{i\}$.  
The event $A_i$ is {\bf mutually independent} of the set of events $\{A_i \,:\, i \in J\}$ if, for every set $J' \subseteq J$,
$$\pr(A_i \cap A_{J'}) = \pr(A_i) \times \pr(A_{J'}),$$
or, equivalently,
$$\pr(A_i\,\,\vline\,\, A_{J'}) = \pr(A_i).$$

The main tool which will be used is the well known and powerful Lov\'asz Local Lemma.

\begin{LLL}[Erd{\H o}s, Lov\'asz \cite{LLL}]
Let $\{A_i : i \in I\}$ be events in a probability space, and for each $A_i$ let $J_i \subseteq I$ be a set of indices such that $A_i$ is mutually independent of $\{A_j \,:\, j \notin J_i \cup \{i\}\}$.  
If there exist real numbers $0 < x_i < 1$ for each $i \in I$ such that $\pr(A_i) < x_i \prod_{j \in J_i}(1-x_j),$ then $$\pr\left( \overline{A}_I \right) \geq \prod_{i \in I}(1-x_i) > 0.$$
\end{LLL}

For an event $A_i$, the set $J_i$ indicated in the Lov\'asz Local Lemma is called the {\bf dependency set} of $A_i$.  
If the maximum size of a dependency set, taken over all $A_i$, is $D$, then setting $x_i = \frac{1}{D+ 1}$ in the Lov\'asz Local Lemma for each $i \in I$ gives the symmetric version of the Local Lemma.

\begin{SLL}[Spencer \cite{Spencer}]
Let $\{A_i : i \in I\}$ be a set of events in a probability space, and for each $A_i$ let $J_i \subseteq I$ be a set of indices such that $A_i$ is mutually independent of $\{A_j \,:\, j \notin J_i \cup \{i\}\}$.  If $|J_i| \leq D$ for all $i \in I$ and $\pr(A_i) <  \frac{1}{e(D + 1)}$ for all $i \in I$, then $\pr\left( \overline{A}_I \right) > 0.$
\end{SLL}

%In some cases, we will need the following modified version of the Local Lemma, due to Spencer \cite{???}.

%\begin{MLL}
%Let $\{A_i : i \in I\}$ be events in a probability space.  Suppose that there exists a real number $p \geq 0$ such that for each $i \in I$ there exists sets $J_i$ and $K_i$ partitioning $I \setminus \{i\}$ in such a way that, for all $K \subseteq K_i,$ 
%$$\pr(A_i \,|\, \overline{A}_K) \leq p.$$
%If $D = \max\{|J_i| : i \in I\}$ and there exists $x \in [0,1)$ such that $x(1 - Dx) \geq p$, then $\pr(\overline{A}_I) > 0.$
%\end{MLL}

%Using the Symmetric Local Lemma, we are able to show that $\chWe(G) \leq 3$ graphs with \mbox{$\delta(G) \geq C\log{\Delta(G)}$} for a constant $C$.  

Let $\{A_e : i \in E(G)\}$ be a set of events in a probability space which are indexed by the edge set of a graph $G$.  We say that an edge $e \in E(G)$ is {\bf covered} by the event $A_{uv}$ if $e$ is incident to either $u$ or $v$, and {\bf uncovered} otherwise.

By applying the Symmetric Local Lemma, a bound for the list variation of Problem \ref{prob2} can be obtained:

\begin{thm}\label{sequencecolouring}
Let $G$ be a graph with minimum degree $\delta$ and maximum degree $\Delta$.  If $\delta > \log_3(2\Delta^2 - 2\Delta + 1) + 2$ then $\chWe(G) \leq 3$.
\end{thm}

\begin{proof}
For each $e \in E(G)$, let $L_e$ be a list of $3$ elements associated with $e$.  Fix an arbitrary ordering of $E(G)$.  Choose $w(e)$ randomly from $L_e$ with uniform probability and let $c(u)$ denote the resulting sequence of weights of edges incident to $u \in V(G)$.  For an edge $uv \in E(G)$, let $A_{uv}$ denote the event that $c(u) = c(v)$; we see that $\pr(A_{uv}) \leq 1/3^{\delta-1}$.

Let $J_{uv} \subset E(G) \setminus \{uv\}$ be the set of edges where $j \in J_{uv}$ if and only if $A_j$ covers $uv$ or an edge incident to $u$ or $v$, and $j \neq {uv}$; % which cover $uv$ as well as all edges incident to $uv$; 
$A_{uv}$ is independent of $\{A_e : e \notin J_{uv} \cup \{uv\} \}$ since no edge incident to $uv$ will have a weight determined by an event $A_e$ with $e \notin J_{uv}$.  Hence, $D = \max\{|J_e| : e \in E(G)\} \leq 2(\Delta-1) + 2(\Delta-1)^2 = 2\Delta(\Delta-1)$.  By the Symmetric Local Lemma, the result holds if 
\begin{eqnarray*}
\dfrac{1}{3^{\delta-1}} < \dfrac{1}{e(2\Delta(\Delta-1)+1)},
\end{eqnarray*}
which is satisfied if $\delta > \log_3(2\Delta^2 - 2\Delta + 1) + 2$.
\end{proof}

If two adjacent vertices have distinct degrees, then their associated sequences will certainly differ.  Hence, regular graphs are of particular interest.  The following corollary is easily obtained from Theorem \ref{sequencecolouring}.

\begin{cor}\label{sequencecolouringreg}
If $G$ is a $d$-regular graph, $d \geq 6$, then $\chWe(G) \leq 3$.
\end{cor}

A similar argument gives $\chWe(G) \leq 4$ if $G$ is $5$-regular, $\chWe(G) \leq 5$ if $G$ is $4$-regular, and $\chWe(G) \leq 6$ if $G$ is $3$-regular.

We now consider total weightings.  Since the List 1-2 Conjecture implies that two weights should suffice for a proper colouring by sums, we consider total $2$-list-weightings.  The upper bound on the probability of a bad event is $1/2^{\delta}$ rather than $1/3^{\delta-1}$; the following bounds are obtained by similar arguments as those used to prove Theorem \ref{sequencecolouring} and Corollary \ref{sequencecolouringreg}.

\begin{thm}\label{totalsequencecolouring}
Let $G$ be a graph with minimum degree $\delta(G) = \delta$ and maximum degree $\Delta(G) = \Delta$.  If $\delta > \log_2(e(2\Delta^2 - 2\Delta + 1))$, then $\chWt(G) \leq 2$.  In particular, if $G$ is $d$-regular for $d \geq 9$, then $\chWt(G) \leq 2$.
\end{thm}

Finally, we examine $\chWe(M)$ and $\chWt(M)$ for a multigraph $M$.  An application of the Local Lemma shows that as long as the maximum edge multiplicity is no more than the minimum degree less a logarithmic term in terms of maximum degree, then the bounds in Conjectures \ref{We} and \ref{Wt} can be obtained for multigraphs.

\begin{thm}\label{multiseq}
Let $M$ be a loopless multigraph with maximum edge multiplicity $\mu(G) = \mu$, minimum degree $\delta(G) = \delta$, and maximum degree $\Delta(G) = \Delta$.
\begin{enumerate}
\item[(1)] If $\mu < \delta -1 - \log_3(2\Delta^2 - 2\Delta + 1) $, then $\chWe(M) \leq 3$.
\item[(2)] If $\mu < \delta - \frac{1}{2} - \log_2(2\Delta^2 - 2\Delta + 1)$, then $\chWt(M) \leq 2$.
\end{enumerate}
\end{thm}

\begin{proof}
(1) Let $L_e$ be a list of $3$ elements associated with the edge $e$.  Fix an arbitrary ordering of $E(G)$.  For an edge $e$, choose its weight $w(e)$ randomly from $L_e$ with uniform probability.   Let $c(u)$ denote the resulting sequence of weights of edges incident to $u \in V(G)$.  For an edge $e=uv$, let $A_e$ denote the event that $c(u) = c(v)$
If $u, v \in V(G)$ are adjacent vertices, and $l$ is the number of edges between them, then $$\pr(A_{uv}) \leq 1/k^{\delta - l} \leq 1/k^{\delta - \mu}.$$
The size of the dependency set $J_e$ is the number of adjacent pairs of vertices from which one vertex is of distance at most one from $u$ or $v$, and hence the maximum size of a dependency set is $D \leq 2\Delta(\Delta-1)$.  By the Symmetric Local Lemma, the first result holds if 
\begin{align*}
\dfrac{1}{3^{\delta - \mu}} < \dfrac{1}{e(2\Delta(\Delta-1) + 1)}.
\end{align*}
(2) Applying the same argument to a random list total weighting from lists of size 2, we need to satisfy the following inequality:
\begin{align*}
\dfrac{1}{2^{\delta - \mu+1}} < \dfrac{1}{2\Delta(\Delta-1) + 1)}. & \qedhere
\end{align*}
\end{proof}

By considering edge $k$-weightings rather than edge $k$-list-weightings one can reduce the size of a bad event's dependency set in graphs with no short cycles, and hence obtain improved bounds.

Given an ordering of the edges of a graph $G$, denote by $e^u_i$ the $i^{\textrm{th}}$ edge incident to $u$ with respect to the ordering of $E(G)$.  A set of events $K \subseteq \{A_e : e \in E(G) \setminus\{uv\}\}$ leaves the edge $uv \in E(G)$ {\bf open} if at least one of $\{e^u_i, e^v_i\} \setminus \{uv\}$ is left uncovered by $K$ for each $1 \leq i \leq \max\{\deg(u), \deg(v)\}$.

\begin{lem}\label{indevents}
Let $G$ be a graph with ordered edge set $E(G) = \{e_1, \ldots, e_m\}$.  Let $w: E(G) \rightarrow [k]$ be a random edge $k$-weighting where, for each $e \in E(G)$, $w(e)$ is chosen with uniform probability from $[k]$; denote by $c(u)$ the resulting sequence of edge weights associated with $u \in V(G)$.  For an edge $uv \in E(G)$, let $A_{uv}$ be the event that $c(u) = c(v)$.  If there exists a set of events $K \subseteq \{A_e : e \in E(G) \setminus\{uv\}\}$ such that $K$ leaves $uv$ open, then $A_{uv}$ is mutually independent of $K$.
\end{lem}  

\begin{proof}
It suffices to prove that $\pr(A_{uv}\,\,\vline\,\, K) = \pr(A_{uv})$, since any proper subset of $K$ leaves more edges adjacent to $uv$ uncovered than does $K$.  If $\deg(u) \neq \deg(v)$, then $\pr(A_{uv}) = \pr(A_{uv}\,\,\vline\,\, K) = 0$.

Assume $\deg(u) = \deg(v) = d$ and suppose that, for some $i$, $uv = e^u_i = e^v_i$.   Clearly \mbox{$\pr(A_{uv}) = \frac{1}{k^{d-1}}$}.  Let \vspace{-0.2in}
\begin{eqnarray*}
U = \{j : e^u_j \textrm{ is uncovered by } K\} \setminus \{uv\} \\
V = \{j : e^v_j \textrm{ is uncovered by } K\}  \setminus \{uv\}.
\end{eqnarray*}
Since $uv$ is left open by $K$, $U \cup V = [d] \setminus \{i\}$.  It follows that
\begin{align*}
\pr(A_{uv}\,\,\vline\,\, K) &= \left( \frac{1}{k^{|U|}} \right)\left( \frac{k^{|U \cap V|}}{k^{|V|}} \right) &= \left( \frac{1}{k^{|U|}}  \right)\left( \frac{1}{k^{|V \setminus U|}} \right)
&= \frac{1}{k^{|U \cup V|}} 
&=  \frac{1}{k^{d-1}} = \pr(A_{uv}).
%\pr(A_{uv}\,\,\vline\,\, A_{K}) &=& \frac{\left( \displaystyle\prod_{j \in J_u} p_j \right)\left(  \displaystyle\prod_{k \in J_v} p_k \right)}{\displaystyle\prod_{t \in J_u \cap J_v} p_t} \\
%&=& \left( \displaystyle\prod_{j \in J_u} p_j \right)\left( \displaystyle\prod_{k \in J_v \setminus J_u} p_k \right)\\
%&=&  \displaystyle\prod_{j \in J_u \cup J_v} p_j \\
%&=& \displaystyle\prod_{j \in [d] \setminus \{i\}} p_j = \pr(A_{uv}) 
\end{align*}

If $uv = e^u_r = e^v_s$ for some $r \neq s$, then $\pr(A_{uv}) = \frac{1}{k^{d-2}}\times\frac{k}{k^3} = \frac{1}{k^d}$.  Again, let $U = \{j : e^u_j \textrm{ is uncovered by } K\} \setminus \{uv\}$ and $V = \{j : e^v_j \textrm{ is uncovered by } K\}  \setminus \{uv\}$.  Since $uv$ is left open by $A_K$, $r \in V$ and $s \in U$, and so
\begin{align*}
\pr(A_{uv}\,\,\vline\,\, K) &= \left( \frac{1}{k^{|U \setminus \{s\}|}} \right)\left( \frac{k^{|U \cap V \setminus \{r,s\}|}}{k^{|V \setminus \{r\}|}} \right) \pr\Big(w(uv) = w(e^u_r) = w(e^v_s) \Big) \\
&= \left( \frac{1}{k^{|U| \setminus \{s\}}}  \right)\left( \frac{1}{k^{|V \setminus U \setminus \{r\}|}} \right) \left( \frac{1}{k^2} \right) \\
&= \frac{1}{k^{|U \cup V \setminus \{r,s\}|}} \left( \frac{1}{k^2} \right) \\
&=  \frac{1}{k^{d-2}} \times \frac{1}{k^2} = \pr(A_{uv}). \hfill \qedhere
\end{align*}
%as desired.
\end{proof}

Note the need for an edge $k$-weighting rather than an edge $k$-list-weighting in Lemma \ref{indevents}; it provides equality between $\pr(A_{uv})$ and  $\frac{1}{k^{d-1}}$ when $uv = e^u_i = e^v_i$ for some index $i$, which is required to show that $\pr(A_{uv}\,\,\vline\,\, K) = \pr(A_{uv})$.

\begin{thm}\label{nolist}
Let $G$ be a graph with minimum degree $\delta(G) = \delta$ and maximum degree $\Delta(G) = \Delta$ and girth at least 5.  If $\delta > \log_3(\Delta^2 - \Delta + 1) + 2$ then $\We(G) \leq 3$.
\end{thm}

\begin{proof}
Fix an arbitrary ordering of $E(G)$ and for an edge $e$, choose its weight $w(e)$ randomly from $\{1,2,3\}$ with uniform probability.   Let $c(u)$ denote the resulting sequence of weights of edges incident to $u \in V(G)$.

For an edge $uv \in E(G)$, let $A_{uv}$ denote the event that $c(u) = c(v)$.  Let $J(uv)$ be the set of edges of distance at most 1 from $u$ not incident to $v$, and let $L(uv) = E(G) \setminus J(uv) \setminus \{uv\}$.  Since the girth of $G$ is at least 5, the distance from $u$ to any end of an edge is $L(uv)$ is at least $2$.  This implies that all edges incident to $u$ except $uv$ are left uncovered by the events $\{A_l : l \in L(uv)\} := K_{uv}$, and hence $uv$ is left open by $K_{uv}$.  By Lemma \ref{indevents}, this implies that $A_{uv}$ is mutually independent of $K_{uv}$; let $J_{uv} = J(uv)$ be the dependency set for $A_{uv}$.  Since the maximum size of a dependency set is $D = \max\{|J_e| : e \in E(G)\} \leq (\Delta-1) + (\Delta-1)^2 = \Delta(\Delta-1)$, 
by the Symmetric Local Lemma the result holds if 
\begin{eqnarray*}
\dfrac{1}{3^{\delta-1}} < \dfrac{1}{e(\Delta(\Delta-1)+1)},
\end{eqnarray*}
which is satisfied if $\delta > \log_3(\Delta^2 - \Delta + 1) + 2$.
\end{proof}

A {\bf $(d,g)$-graph} is a $d$-regular graph with girth $g$.
Theorem \ref{nolist} implies that, for most $(d,g)$-graphs, three edge weights suffice for adjacent vertices to receive distinct sequences.

\begin{cor}
If $G$ is a $(d,g)$-graph with $d \neq 4$ and $g \geq 5$, then $\We(G) \leq 3$.
\end{cor}

\begin{proof}
If $d = 3$, then $G$ is $3$-colourable (by Brook's Theorem).  In \cite{KLT04} it is shown that if $G$ is complete or 3-colourable then $\Se(G) \leq 3$, and so certainly $\We(G) \leq 3$.  If $d \geq 5$, then $d > \log_3(d^2 - d + 1) + 2$, and so we may apply Theorem \ref{nolist}.
\end{proof}

Recall that $\me(G) \leq 3$ if $\delta(G) \geq 1000$.  As such, the only graphs for which it remains to show that $\We(G) \leq 3$ holds are those with small minimum degree (at most 1000) and comparatively large maximum degree ($\Omega(3^{\delta(G)/2})$).

The other theorems from this section have similar relaxations.

\begin{thm}
Let $G$ be a graph with minimum degree $\delta(G) = \delta$, maximum degree $\Delta(G) = \Delta$, and girth at least 5.  If $\delta > \log_2(e(\Delta^2 - \Delta + 1))$, then $\Wt(G) \leq 2$.  In particular, if $G$ is $d$-regular for $d \geq 7$, then $\Wt(G) \leq 2$.
\end{thm}

\begin{thm}
Let $M$ be a loopless multigraph with maximum edge multiplicity $\mu(G) = \mu$, minimum degree $\delta(G) = \delta$, maximum degree $\Delta(G) = \Delta$, and girth at least 5.
\begin{enumerate}
\item[(1)] If $\mu < \delta -1 - \log_3(\Delta^2 - \Delta + 1) $, then $\We(M) \leq 3$.
\item[(2)] If $\mu < \delta - \frac{1}{2} - \log_2(\Delta^2 - \Delta + 1)$, then $\Wt(M) \leq 2$.
\end{enumerate}
\end{thm}

\section{Sequence irregularity strength}

The {\bf irregularity strength} of a graph $G$ is the smallest integer $k$ such that $G$ has an edge $k$-weighting giving every vertex in $G$ a distinct sum of incident edge weights.  This well studied graph parameter was introduced by Chartrand et al. in \cite{CJL+} where it was denoted $\textup{s}(G)$.  In keeping with our notation, we denote the irregularity strength of $G$ as $\textup{s}_{\Sigma}^e(G)$.  Many variations of irregularity strength have been studied, including (but not limited to) requiring all vertices to receive distinct multisets, products, or sets of incident edge weights rather than distinct sums.  These parameters are called the multiset irregularity strength, product irregularity strength and set irregularity strength, and they are denoted $\textup{s}_m^e(G)$, $\textup{s}_{\Pi}^e(G)$, and $\textup{s}_s^e(G)$ respectively.   Note that a graph must be nice for these parameters to be well defined.  Kalkowski, Karo\'nski, and Pfender \cite{KKP3} show that $\textup{s}_{\Sigma}^e(G) \leq \lceil 6n/\delta \rceil$ for every nice graph $G$.  Aigner et al. \cite{ATT} show that if $G$ is a $d$-regular graph $d \geq 2$, then $\textup{s}_{m}^e(G) \leq (5e(d+1)!n)^{1/d}$.  Burris and Schelp \cite{BS97} show that $\textup{s}_s^e(G) \leq C_{\Delta}\textup{max}\{n_i^{1/i} \,:\, 1 \leq i \leq \Delta(G)\}$, where $C_{\Delta}$ is a constant relying only on $\Delta$ and $n_i$ denotes the number of vertices of degree $i$ in $G$ (in fact, their edge weighting gives a proper edge-colouring as well).  Only partial results are known for $\textup{s}_{\Pi}^e(G)$.

The {\bf specific sequence irregularity strength} of $G$, denoted $\textup{s}_{{\sigma^*}}^e(G)$, is the smallest $k$ such that there exists an ordering of $E(G)$ and an edge $k$-weighting of $G$ such every vertex receives a distinct induced sequence of incident edge weights.  The {\bf general sequence irregularity strength} of $G$, denoted $\textup{s}_{\sigma}^e(G)$, is the smallest $k$ such that for every ordering of $E(G)$ there exists an edge $k$-weighting of $G$ such every vertex receives a distinct induced sequence of incident edge weights.

Each ``irregularity strength type'' parameter has the usual natural list variant -- rather than each edge receiving a weight from $\{1,2,\ldots,k\}$, each is weighted from its own independently assigned list of $k$ weights. The {\bf general sequence list-irregularity strength} of a graph $G$ is denoted $\textup{ls}_{\sigma}^e(G)$; the other parameters are extended similarly.  As with the 1-2-3 Conjecture variations, one could weight the vertices of $G$ as well as the edges; the corresponding parameters have ``t'' in place of ``e'' in the superscript (e.g. $\textup{ls}_{\sigma}^t(G)$ for total list-weightings which distinguish vertices by sequences for any ordering of $E(G)$).

Let $M_G := \textup{max}\{\lceil n_i^{1/i} \rceil : 1 \leq i \leq \Delta(G)\}$.  Clearly $\textup{s}_{{\sigma^*}}^e(G) \geq M_G$, since any valid weighting from $\{1,2,\ldots,k\}$ must satisfy $k^d \geq n_d$ for each degree $d$.  We make the following conjectures, which motivates the results that follow:

\begin{conj}
If $G$ is a nice graph, then $\textup{s}_{\sigma}^e(G) = M_G$.
\end{conj}

\begin{conj}
If $G$ is a nice graph, then $\textup{ls}_{\sigma}^e(G) = M_G$.
\end{conj}

The aforementioned bound on $\textup{s}_{s}^e(G)$ shows that there is a constant $C$ such that \mbox{$\textup{s}_{\sigma^*}^e(G) \leq \textup{s}_{\sigma}^e(G) \leq CM_G$}.  The bound on $\textup{s}_{m}^e(G)$ stated above gives a similar result for $d$-regular graphs.  In fact, it follows quite easily from the proof details of Aigner et al.~\cite{ATT} that their bound holds for $\textup{ls}_{m}^e(G)$.  By directly considering colouring by sequences, these bounds can be further improved.

\begin{thm}\label{dregseqirr}
If $G$ is a nice $d$-regular graph, then $\textup{ls}_\sigma^e(G) \leq \left\lceil \left(2e(d+1)(n-d)\right)^{1/d-1} \right\rceil$.
\end{thm}

\begin{proof}
Fix an arbitrary ordering of $E(G)$.  Let $L_e$ be a set of size $k =  \left\lceil \left(2e(d+1)(n-d)\right)^{1/d-1} \right\rceil$ associated with the edge $e$; choose its weight $w(e)$ randomly from $L_e$ with uniform probability.   Let $c(u)$ denote the resulting sequence of weights of edges incident to $u \in V(G)$.

For an edge $e=uv$, let $A_e$ denote the event that $c(u) = c(v)$.  By the same argument in the proof of Theorem \ref{sequencecolouring}, $\pr(A_e) \leq 1/k^{d-1}$.  For a non-adjacent pair of vertices $p = \{u,v\}$, $\pr(A_p) \leq 1/k^d$ where $A_p$ is the event that $c(u) = c(v)$.

The size of a dependency set $J_e$ for an edge $e=uv$ is the number of edges of distance at most one from $e$ plus the number of nonadjacent pairs of vertices containing $u$, $v$, or a neighbour of $u$ or $v$; in other words, the total number of pairs of vertices containing at least one vertex in $N(u) \cup N(v)$.  Hence,
$$|J_e| \leq {n \choose 2} - {n-2d \choose 2} = d(2n - 2d - 1).$$
Similarly, the size of $J_p$ is
$$|J_p| \leq {n \choose 2} - {n-2d-2 \choose 2} = (d+1)(2n - 2d - 3).$$
The probability of a bad event $A \in \{A_e, A_p \,:\, e \in E(G), p \in (V(G) \times V(G)) \setminus E(G)\}$ is
\begin{eqnarray*}
\pr(A) \leq \dfrac{1}{k^{d-1}} \leq \dfrac{1}{2e(d+1)(n-d)} < \dfrac{1}{e(\max\{|J_e|, |J_p|\} + 1)},
\end{eqnarray*}
and so the result holds by the Symmetric Local Lemma.
\end{proof}

A bound for total list weightings is similarly obtained:

\begin{thm}\label{dregseqirrtot}
For any $d$-regular graph $G$, $d \geq 2$, $\textup{ls}_\sigma^t(G) \leq \left\lceil \left(2e(d+1)(n-d)\right)^{1/d} \right\rceil$.
\end{thm}

As with our results on $\chWe(G)$ and $\chWt(G)$, these theorems generalize to graphs with arbitrary maximum and minimum degrees.  In particular, we can show that there is a constant bound on general sequence list irregularity strength for graphs with sufficiently large minimum degree.

\begin{thm}\label{seqirrstrengthgeneral}
If $G$ is a graph with minimum degree $\delta(G) = \delta$ and maximum degree $\Delta(G) = \Delta$, then 
\begin{eqnarray*}
\textup{ls}_\sigma^e(G) \leq \left\lceil\left(2e(\D+1)(n-\D)\right)^{1/\delta-1} \right\rceil, \textrm{ and  \,\,}
\textup{ls}_\sigma^t(G) \leq \left\lceil \left(2e(\D+1)(n-\D)\right)^{1/\delta} \right\rceil.
\end{eqnarray*}
\end{thm}

As a consequence, there is a constant bound on general sequence list irregularity strength for graphs with sufficiently large minimum degree.
%
%\begin{cor}\label{largemin}
%Let $n,k \in \Z^{+}$.  If $G$ is a graph on $n$ vertices with %maximum degree $\Delta(G) = \Delta$ and 
%minimum degree $\delta(G) = \delta >  \log_k\left(\frac{e}{2}(n+1)^2 + 1\right) + 1$, then $\textup{ls}_\sigma^e(G) \leq k$.
%\end{cor}

\begin{cor}\label{largemin}
Let $n,k \in \Z^{+}$.  If $G$ is a graph on $n$ vertices with %maximum degree $\Delta(G) = \Delta$ and 
minimum degree $\delta(G) = \delta > c\log{n}$ for large enough $c=c(k)$, then $\textup{ls}_\sigma^e(G) \leq k$.
\end{cor}

\begin{proof}
Choose $c$ so that $c\log{n} \geq \log_k\left(\frac{e}{2}(n+1)^2 + 1\right) + 1$.
Note that the function $f(\Delta) = (\D+1)(n-\D)$ is maximized when $\Delta = \frac{1}{2}(n-1)$, and so $(\D+1)(n-\D) \leq \frac{1}{4}(n+1)^2$.  
Since $\delta(G) >  \log_k\left(\frac{e}{2}(n+1)^2 + 1\right) + 1$,
\begin{align*}
& k^{\delta-1} > \left(\frac{e}{2}(n+1)^2 + 1\right) \\
 \implies & k >  \left(2e(\D+1)(n-\D)\right)^{1/\delta-1} \\
 \implies & k \geq  \left\lceil\left(2e(\D+1)(n-\D)\right)^{1/\delta-1} \right\rceil.
\end{align*}
The result follows by Theorem \ref{seqirrstrengthgeneral}.
\end{proof}

Finally, we extend our irregularity strength results to multigraphs; the proof follows similarly to that of Theorem \ref{multiseq}.

\begin{thm}
Let $M$ be a loopless multigraph with maximum edge multiplicity $\mu(G) = \mu$, minimum degree $\delta(G) = \delta$, and maximum degree $\Delta(G) = \Delta$.  For any positive integer $k$,
\begin{enumerate}
\item[(1)] if $\mu < \delta - \log_k(2e(\D+1)(n-\D)) $, then $\textup{ls}_{\sigma}^e(M) \leq k$.
\item[(2)] if $\mu < \delta - \log_k(2e(\D+1)(n-\D)) + 1$, then $\textup{ls}_{\sigma}^t(M)  \leq k$.
\end{enumerate}
\end{thm}

\section{Acknowledgements}

The authors express their gratitude to their respective funding agencies - Carleton University, the Natural Sciences and Engineering Research Council of Canada, and the Ontario Ministry of Research and Innovation.

\bibliographystyle{plain}
\bibliography{references}

\end{document}